\documentclass{amsart}[12pt]

\newtheorem{theorem}{Theorem}[section]
\newtheorem{lemma}[theorem]{Lemma}
\newtheorem{proposition}[theorem]{Proposition}
\newtheorem{corollary}[theorem]{Corollary}

\theoremstyle{definition}
\newtheorem{definition}[theorem]{Definition}

\newtheorem{remark}[theorem]{Remark}

\usepackage{amscd,amssymb}

\begin{document}

\title
{$\mathbb{Z}_{3}$-graded identities of the pair $(M_{3}(K),gl_{3}(K))$}
\author{ Lu\'is Felipe Gon\c{c}alves Fonseca}

\address{Lu\'is Felipe Gon\c{c}alves Fonseca: Instituto de Ci\^encias Exatas e Tecnl\'ogicas \\
Universidade Federal de Vi\c{c}osa \\
Florestal, MG, Brazil. ZIP Code: 35690-000.}
\email{luisfelipe@ufv.br}

\date{}

\maketitle

\begin{abstract}
Let $M_{n}(K)$ be the algebra of $n \times n$ matrix over an infinite integral domain $K$. Let $gl_{n}(K)$ be the Lie algebra of $n \times n$ matrix with the usual Lie product over $K$. Let $G = \{g_{1},\ldots,g_{n}\}$ be a group of order $n$. We describe the polynomials that form a basis for the $G$-graded identities of the pair $(M_{n}(K),gl_{n}(K))$ with an elementary $G$-grading induced by the $n$-tuple $(g_{1},\ldots,g_{n})$. In the end, we describe an explicit basis for the $\mathbb{Z}_{3}$-graded identities of the pair $(M_{3}(K),gl_{3}(K))$.
\end{abstract}

\section{Introduction}

In 1973, Razmyslov (\cite{Razmyslov}) introduced the weak polynomial identities. In this paper, the author described a basis for the associative matrix algebra of order $2$ and the Lie algebra of matrix algebra of order $2$ with zero trace.
Formanek (\cite{Formanek}) and Razmyslov (\cite{Razmyslov2}) solved a Kaplanky's problem about the existence of multihomogeneous central polynomial of matrix algebra of order $n \geq 3$. Razmyslov used notions of weak polynomial identities to get a central polynomial.
Since Razmyslov, the weak identities have been studied by many mathematicians (\cite{Drensky-Rashkova},\cite{Di Vincenzo2}, \cite{Koshlukov}, \cite{Koshlukov-Krasilnikov}).

An important breakthrough occurred in the late 1980s in PI-Theory. Kemer (\cite{Kemer}) solved a proposed question of Specht about T-ideals over the fields of characteristic zero \cite{Specht}. To solve Specht's problem, the $\mathbb{Z}_{2}$-graded identities was an important tool.

In 1992, Di Vincenzo (\cite{Di Vincenzo}) described the $\mathbb{Z}_{2}$-graded identities of $2 \times 2$ matrix algebra over a field of characteristic zero. After this pioneer work, many authors contributed to the study of $n \times n$ matrix algebra over an infinite field. For fields of characteristic zero, we can cite Vasilovsky (\cite{Vasilovsky1}, \cite{Vasilovsky2}), Bahturin and Drensky \cite{Bahturin-Drensky}. For arbitrary infinite fields, we can refer Azevedo (\cite{Azevedo1}, \cite{Azevedo2}) and Silva (\cite{Silva}).

Let $G = \{g_{1},\ldots,g_{n}\}$ be a group of order $n$ and let $M_{n}(K)$ be the algebra of $n \times n$ matrix over an infinite integral domain $K$. Let $gl_{n}(K)$ be the Lie algebra of $n \times n$ matrix with the usual Lie product. In this paper, we combine the methods and ideas of \cite{Fonseca-Mello} (Section \ref{Generic matrices}), \cite{Silva} (Sections \ref{Generic matrices} and \ref{Preliminary Results}).

We describe the polynomials that form a basis for the $G$-graded identities of the pair $(M_{n}(K),gl_{n}(K))$ with an elementary $G$-grading induced by the $n$-tuple $(g_{1},\ldots,g_{n})$. Furthermore, we describe an explicit basis for the $\mathbb{Z}_{3}$-graded identities of the pair $(M_{3}(K),gl_{3}(K))$.

\section{Preliminaries}

Let $K$ be an infinite integral domain. Let $\mathbb{N}$ be the set of positive integers. Let $n \in \mathbb{N}$. We denote the set of the first $n$ positive integers by $\widehat{n}$ .

Let $A$ be an associative algebra over $R$ and let $A^{(-)}$ be the Lie algebra associated with $A$. The Lie multiplication of $A^{(-)}$ is given by $[a,b] = ab - ba$ for all $a,b \in A$. We also write $[a,b,c] = [[a,b],c]$ for all $a,b,c \in A$.

Let $G = \{g_{1},\ldots,g_{n}\}$ be a group of order $n$ and let $1_{G}$ be the identity of $G$.

An algebra $A$ is $G$-graded if $A$ can be written as direct sum of left $R$-submodules $A = \bigoplus_{g \in G} A_{g}$ such that $A_{g}A_{h} \subset A_{gh}$ for all $g,h \in G$. The elements of $A_{g}$ are said to be homogeneous of $G$-degree $g$. The component $A_{1_{G}}$ is said to be neutral. Let $a \in A$. We say that $\alpha(a) = g$ if $a \in A_{g}$. If $B$ is a Lie subalgebra of $A^{(-)}$ such that $B = \bigoplus_{g \in G} B_{g}, B_{g} = B\cap A_{g} (g \in G)$, we say that $(A,B)$ is a graded pair.

Let $M_{n}(K)$ be the algebra of $n\times n$ matrices over $K$, and let $gl_{n}(K) = (M_{n}(K))^{(-)}$. We denote the matrix unit in which only non-zero entry is $1$ in the $i$-th row and $j$-th column by $e_{ij} \in M_{n}(K)$. Let $\overline{g} = (h_{1},\ldots,h_{n}) \in G^{n}$. We say that $M_{n}(K)$ is equipped with the elementary grading induced by $\overline{g}$ when each $e_{ij}$ is homogeneous and $\alpha(e_{ij}) = h_{i}^{-1}h_{j}$. The following proposition is a well-known result about graded matrix algebra over fields.

\begin{proposition}(\cite{Dascalescu})
Let $F$ be a field. Let $G$ be a group. The $G$-grading of $M_{n}(F)$ is elementary if and only if all matrix units $e_{ij}$ are homogeneous.
\end{proposition}

An endomorphism $\phi: A \rightarrow A$ is called graded when $\phi(A_{g}) \subset A_{g}$ for all $g \in G$. A subalgebra $B \subset A$ is said to be graded when $B = \bigoplus_{g \in G} B_{g}$, where $B_{g} = A \cap A_{g}$ for all $g \in G$.

Let $\{X_{g}| g \in G\}$ be a family of disjoint infinite countable sets indexed by $G$. Let $X = \bigcup_{g \in G}\{X_{g}\}$. We denote the free associative algebra freely generated by $X$ by $K\langle X \rangle$. Purely for convenience, we denote the elements of $X$ by the letters $x,x_{1},\ldots,x_{i},\ldots$. We denote the subalgebra of $K\langle X \rangle^{-}$ generated by $X$ by $L$.

Let $m = x_{i_{1}}\ldots x_{i_{l}} \in K\langle X \rangle$ be a monomial. We define $\alpha(m) = \alpha(x_{i_{1}}).\cdots.\alpha(x_{i_{l}})$ and $\alpha(1) = 1_{G}$, where $1$ is the unity of $K\langle X \rangle$. The algebra $K\langle X \rangle = \bigoplus_{g \in G} (K\langle X \rangle)_{g}$ is a graded algebra, where
\begin{center}
$(K\langle X \rangle)_{g} = span_{K}\{x_{i_{1}}\ldots x_{i_{l}}| \alpha(x_{i_{1}}).\cdots.\alpha(x_{i_{l}}) = g \}$.
\end{center}

A polynomial $f(x_{1},\ldots,x_{m}) \in K\langle X \rangle$ is called a graded polynomial identity of the pair $(A,B)$ if $f(b_{1},\ldots,b_{m}) = 0$ for all $b_{1} \in B_{\alpha(x_{1})},\ldots,b_{m} \in B_{\alpha(x_{m})}$. Analogously, we can define a graded polynomial identity of $A$. We denote the set of all graded polynomial identities of the algebra $A$ and the pair $(A,B)$ by $T_{G}(A)$ and $T_{G}(A,B)$, respectively. Notice that $T_{G}(A) = T_{G}(A,B)$ when $B = A^{(-)}$.

An ideal $I \subset K\langle X \rangle$ is said to be a $T_{G}$-ideal if $\phi(I) \subset I$ for all graded endomorphisms of $K\langle X \rangle$. An ideal $I \subset K\langle X\rangle$ is called a weak $T_{G}$-ideal if $\phi(I) \subset I$ for all graded endomorphism $\phi$ of $K\langle X \rangle$ such that
\begin{center}
$\phi(x) \subset L \cap K\langle X \rangle_{\alpha(x)}$ for all $x \in X$.
\end{center}

Let $S \subset K\langle X \rangle$ be a non-empty set of polynomials. We define $\langle S \rangle$ to be the intersection of all $T_{G}$-ideals that contain $S$. In the same way, we define $\langle S \rangle_{w}$ to be the intersection of all weak $T_{G}$-ideals that contain $S$.

A set $S \subset K\langle X \rangle$ is a basis for the graded identities of $A$ if $T_{G}(A) = \langle S \rangle$. Similarly, $S$ is a basis for the graded identities of the pair $(A,B)$ if $T_{G}(A,B) = \langle S \rangle_{w}$. In this situation, we say that the graded identities of the pair $(A,B)$ follow from $S$. Note that if $S$ is a basis for $T_{G}(A,A^{(-)})$, then $S$ is a basis for $T_{G}(A)$ too.

\section{Generic matrices}\label{Generic matrices}

Let $Dom(g) = \{i \in \widehat{n} |g_{i}g \in G\}$ and $Im(g) = \{j \in \widehat{n} | g_{j}g^{-1} \in G \}$. It is clear that $Dom(g) = Im(g) = \widehat{n}$. Let $\phi_{g}: \widehat{n} \rightarrow \widehat{n}$ be the function defined by the following rule: $\phi_{g}(i) = j$ if $g_{j}=g_{i}g$. As $Dom(g) = Im(g) = \widehat{n}$, we have $\phi_{g}$ is a bijection for all $g \in G$.

\begin{remark}
Let $g_{j},g_{k} \in G$. Note that if $\phi_{g_{j}}(i) = \phi_{g_{k}}(i)$ for all $i \in \widehat{n}$, then $g_{j} = g_{k}$. Furthermore, it is clear
that $\phi_{g_{j}}(\phi_{g_{k}}(i)) = \phi_{g_{j}g_{k}}(i)$ for all $i \in \widehat{n}$.
	
\end{remark}


Let $\Omega = \{y_{i,\phi_{g}(i)}^{k}| g \in G, i \in \widehat{n}, k \in \mathbb{N}\}$. We denote the set of commuting polynomials in $\Omega$ by $K[\Omega]$. Let $M_{n}(\Omega)$ be the algebra of $n \times n$ matrix over $K[\Omega]$.

\begin{definition}\label{generic}
The generic matrix algebra is the subalgebra of $M_{n}(\Omega)$ generated by the following elements:
\begin{center}
$A_{k,g} = \sum_{i=1}^{n} y_{i,\phi_{g}(i)}^{k} e_{i \phi_{g}(i)}, g \in G, k = 1,2,3,\ldots$.
\end{center}
We denote this subalgebra by $Gen$.

We say that $y_{i,\phi_{g}(i)}^{k}$ is the commuting variable associated with the matrix unit $e_{i \phi_{g}(i)}$.
\end{definition}

Just as $M_{n}(K)$, we can equip $M_{n}(\Omega)$ and $Gen$ with an elementary grading induced by $\overline{g}$.

\begin{proposition}\label{matrizesgenericas}
The pairs $(M_{n}(K),gl_{n}(K))$ and $(Gen, Gen^{(-)})$ have the same graded identities.
\end{proposition}

\begin{proof}
See for instance (\cite{Drensky-Formanek}, Proposition 1.3.2, page 13). The proof is analogous.
\end{proof}

\begin{remark}\label{operacoes}
Let $A_{k,h_{1}} = \sum_{i=1}^{n} y_{i,\phi_{h_{1}}(i)}^{k} e_{i \phi_{h_{1}}(i)}$ and $A_{l,h_{2}} = \sum_{i=1}^{n} y_{i,\phi_{h_{2}}(i)}^{l} e_{i \phi_{h_{2}}(i)}$ be two generic matrices. Recall that $e_{i \phi_{h_{1}}(i)}e_{j \phi_{h_{2}}(j)} \neq 0$ if and only if $j = \phi_{h_{1}}(i)$. If the product is non zero, then the product result is $e_{i (\phi_{h_{2}h_{1}}(i))}.$ At the light of the bijection of $\phi_{g}$ for all $g \in G$, we can conclude that
\begin{center}
$A_{k,h_{1}}A_{l,h_{2}} = \sum_{i=1}^{n} y_{i,\phi_{h_{1}}(i)}^{k} y_{\phi_{h_{1}}(i),(\phi_{h_{2}h_{1}}(i))}^{l} e_{i (\phi_{h_{2}h_{1}}(i))} $.
\end{center}
\end{remark}

\begin{definition}
Let $\overline{h} = (h_{1},\ldots,h_{q}) \in G^{q}$. For each $t \in \{0,\ldots,q-1\}$, let us define $\beta_{t}(\overline{h}) = \phi_{h_{1+t}.\cdots.h_{1}}$.
\end{definition}

For simplify, in the next proposition, we denote $(\beta_{t}(\overline{h}))(i)$ by $\beta_{t}(\overline{h})(i)$.

\begin{proposition}\label{monomio}
Let $m(x_{1},\ldots,x_{k}) = x_{i_{1}}\ldots x_{i_{q}}$ be a monomial such that
\begin{center}
$\overline{h} = (h_{i_{1}},\ldots,h_{i_{q}}) = (\alpha(x_{i_{1}}),\ldots,\alpha(x_{i_{q}}))$.
\end{center}
Then $m(A_{1,h_{1}},\ldots,A_{k,h_{k}})$ is equal to
\begin{center}
$\sum_{j = 1}^{n} y_{j,\beta_{0}(\overline{h})(j)}^{i_{1}}y_{\beta_{0}(\overline{h})(j),\beta_{1}(\overline{h})(j)}^{i_{2}}. \cdots .y_{\beta_{q-2}(\overline{h})(j),\beta_{q-1}(\overline{h})(j)}^{i_{q}} e_{j\beta_{q-1}(\overline{h})(j)}$.
\end{center}
\end{proposition}
\begin{proof}
We prove by induction on $q$. If $q = 1$, the result follows from Definition \ref{generic}. Suppose that the result is valid for $k = q-1$. Hence $m(A_{1,h_{1}},\ldots,A_{k,h_{k}})$ is equal to

\begin{center}
$\sum_{j = 1}^{n} y_{j,\beta_{0}(\overline{h})(j)}^{i_{1}}y_{\beta_{0}(\overline{h})(j),\beta_{1}(\overline{h})(j)}^{i_{2}} \cdots y_{\beta_{q-3}(\overline{h})(j),\beta_{q-2}(\overline{h})(j)}^{i_{q-1}} e_{j\beta_{q-2}(\overline{h})(j)}$
\end{center}
and
\begin{center}
$A_{i_{q},h_{i_{q}}} = \sum_{j=1}^{n} y_{j,\phi_{h_{i_{q}}}(j)}^{i_{q}} e_{j \phi_{h_{i_{q}}}(j)}$.
\end{center}

Following the idea of Remark \ref{operacoes}, we can conclude that $m(A_{1,h_{1}},\ldots,A_{k,h_{k}})$ is
\begin{center}
$\sum_{j = 1}^{n} y_{j,\beta_{0}(\overline{h})(j)}^{i_{1}}y_{\beta_{0}(\overline{h})(j),\beta_{1}(\overline{h})(j)}^{i_{2}}. \cdots . y_{\beta_{q-2}(\overline{h})(j),\beta_{q-1}(\overline{h})(j)}^{i_{q}} e_{j\beta_{q-1}(\overline{h})(j)}$.
\end{center}
The proof is complete.
\end{proof}

\begin{corollary}\label{identidadesmonomiais}
Let $m \in K\langle X \rangle$ be a monomial. Then $m \notin T_{G}(M_{n}(K),gl_{n}(K))$.
\end{corollary}


\begin{proposition}\label{proposicaochave}
Let $m(x_1,\dots,x_l) = x_{i_{1}}.\cdots.x_{i_{q}}$ and $n(x_1,\dots,x_l) = x_{j_{1}}.\cdots.x_{j_{q}}$ be two monomials such that the matrices
\begin{center}
$n(A_{1,\alpha(x_{1})},\dots,A_{l,\alpha(x_{l})})$ and $m(A_{1,\alpha(x_{1})},\dots,A_{l,\alpha(x_{l})})$
\end{center}
have in the same position the same non-zero entry.

There exist matrix units $e_{a_{1}b_{1}} \in (M_{n}(K))_{\alpha(x_{i_{1}})},\ldots,e_{a_{q}b_{q}} \in (M_{n}(K))_{\alpha(x_{i_{q}})}, \newline e_{c_{1}d_{1}} \in (M_{n}(K))_{\alpha(x_{j_{q}})},\ldots,e_{c_{q}d_{q}} \in (M_{n}(K))_{\alpha(x_{j_{q}})}$ such that

\begin{center}
$e_{a_{1}b_{1}}.\cdots.e_{a_{q}b_{q}} = e_{c_{1}d_{1}}.\cdots.e_{c_{q}d_{q}}$.
\end{center}

Furthermore, there exists a permutation $\sigma \in S_{q}$ such that $e_{a_{\sigma(h)}b_{\sigma(h)}} = e_{c_{h}d_{h}}$ for all $h \in \widehat{q}$.
\end{proposition}
\begin{proof}
Let us define:

\begin{flushleft}
 $\overline{h} = (\alpha(x_{i_{1}}),\ldots,\alpha(x_{i_{q}})) = (h_{1},\ldots,h_{q})$, $\overline{h'} = (\alpha(x_{j_{1}}),\ldots,\alpha(x_{j_{q}})) = (h_{1}',\ldots,h_{q}')$.
\end{flushleft}

According to the hypothesis, there exists $i \in \widehat{n}$ such that

\begin{center}
 $y_{i,\beta_{1}(\overline{h})(i)}^{i_{1}}y_{\beta_{1}(\overline{h})(i),\beta_{2}(\overline{h})(i)}^{i_{2}}. \cdots. y_{\beta_{q-2}(\overline{h})(i),\beta_{q-1}(\overline{h})(i)}^{i_{q}} e_{i\beta_{q-1}(\overline{h})(i)} =$
\end{center}
\begin{center}
 $y_{i,\beta_{1}(\overline{h'})(i)}^{j_{1}}y_{\beta_{1}(\overline{h'})(i),\beta_{2}(\overline{h'})(i)}^{j_{2}}. \cdots. y_{\beta_{q-2}(\overline{h'})(i),\beta_{q-1}(\overline{h'})(i)}^{j_{q}} e_{i\beta_{q-1}(\overline{h'})(i)}$.
\end{center}
Besides that $e_{i,\beta_{1}(\overline{h})(i)}e_{\beta_{1}(\overline{h})(i),\beta_{2}(\overline{h})(i)}. \cdots. e_{\beta_{q-2}(\overline{h})(i),\beta_{q-1}(\overline{h})(i)}$ is equal to
\begin{center}
$e_{i,\beta_{1}(\overline{h'})(i)}e_{\beta_{1}(\overline{h'})(i),\beta_{2}(\overline{h'})(i)}. \cdots. e_{\beta_{q-2}(\overline{h'})(i),\beta_{q-1}(\overline{h'})(i)}$.
\end{center}

So, for each $r \in \widehat{q}$, there exists a unique $s = \sigma(r) \in S_{q}$ such that
\begin{center}
$y_{\beta_{s-1}(\overline{h})(i),\beta_{s}(\overline{h})(i)}^{i_{s}} =
y_{\beta_{r-1}(\overline{h'})(i),\beta_{r}(\overline{h'})(i)}^{j_{r}}$,

$e_{\beta_{s-1}(\overline{h})(i)\beta_{s}(\overline{h})(i)} = e_{\beta_{r-1}(\overline{h'})(i)\beta_{r}(\overline{h'})(i)}$.
\end{center}
\end{proof}

\begin{remark}\label{comentario}
Let $m$ and $n$ be the monomials as given in Proposition \ref{proposicaochave} statement. If $x_{i_{1}} \neq x_{j_{1}}$, then
$y_{i,\beta_{1}(\overline{h})(i)}^{i_{1}} \neq y_{i,\beta_{1}(\overline{h'})(i)}^{j_{1}}$. Consequently, we have
$\sigma(1) > 1$. Now, let $k,l \in \widehat{q}$. Notice also that if $y_{\beta_{k-1}(\overline{h})(i),\beta_{k}(\overline{h})(i)}^{i_{k}} =
y_{\beta_{l-1}(\overline{h'})(i),\beta_{l}(\overline{h'})(i)}^{j_{l}}$, then $x_{i_{k}} = x_{j_{l}}$.
\end{remark}

The next proposition is a straightforward consequence of definition of $Gen$ and Proposition \ref{monomio}.

\begin{proposition}\label{multi}
Let $f$ be a weak graded polynomial identity of $Gen$.

Let $f = \sum_{i=1}^{l} f_{i}$ be a decomposition of $f$ in multi-homogeneous summands of different multi-degree. Then each $f_{i}$ is a weak graded polynomial identity of $Gen$.
\end{proposition}

\section{Graded identities of the pair $(M_{n}(K),gl_{n}(K))$}

\begin{proposition}\label{identities}
The following polynomials are graded identities of the pair \newline $(M_{n}(K),gl_{n}(K))$.

\begin{enumerate}
\item $[x_{1},x_{2}], \ \ \alpha(x_{1}) = \alpha(x_{2}) = 1_{G}.$
\item $x_{1}x_{2}x_{3} - x_{3}x_{2}x_{1},\ \ \alpha(x_{1}) = \alpha(x_{3}) = (\alpha(x_{2}))^{-1}.$
\end{enumerate}

\end{proposition}
\begin{proof}
It follows from (\cite{Bahturin-Drensky}, Lemma 4.1).
\end{proof}

\begin{definition}\label{identidadesessenciais}
Let $J$ be the weak $T_{G}$-ideal generated by the following identities
\begin{enumerate}
\item (Identities of type 1) $[h_{1},h_{2}]$, where $h_{1},h_{2}$ are monomials, $\alpha(h_{1}) = \alpha(h_{2}) = 1_{G}$, and $h_{1}h_{2}$ is multi-linear.
\item (Identities of type 2) $h_{1}h_{2}h_{3} - h_{3}h_{2}h_{1}$, where $h_{1},h_{2},h_{3}$ are monomials, $\alpha(h_{1}) = \alpha(h_{3}) = (\alpha(h_{2}))^{-1}$, and $h_{1}h_{2}h_{3}$ is multi-linear.
\end{enumerate}
\end{definition}

\begin{definition}
We denote the weak $T_{G}$-ideal generated by the identities of type 1 by $J_{1}$. We denote the weak $T_{G}$-ideal generated by the identities of type 2 by $J_{2}$.
\end{definition}

Notice that $J = J_{1} + J_{2}$. Let $S_{1} \subset J_{1}$, and let $S_{2} \subset J_{2}$. If $J_{1} = \langle S_{1} \rangle_{w}, J_{2} = \langle S_{2} \rangle_{w}$, then $J = \langle S_{1}, S_{2} \rangle_{w}$.

\begin{remark}\label{atencao}
Let us consider $H_{1,2} = [h_{1},h_{2}], H_{3,4,5} = h_{3}h_{4}h_{5} - h_{5}h_{4}h_{3}$,
where $h_{1},h_{2}, h_{3}, h_{4}, h_{5}$ are monomials,
$\alpha(h_{1}) = \alpha(h_{2}) = 1_{G}$ and $\alpha(h_{3}) = \alpha(h_{5}) = (\alpha(h_{4}))^{-1}$.
The polynomials $H_{1,2}$ and $H_{3,4,5}$ are graded identities of the pair $(M_{n}(K),gl_{n}(K))$. Consider convenient weak graded endomorphisms of $K\langle X \rangle$, we can prove that
\begin{center}
$H_{1,2} \in J_{1}$, and $H_{3,4,5} \in J_{2}$.
\end{center}
By modulo $J$, we have $h_{1}h_{2} \equiv h_{2}h_{1}, h_{3}h_{4}h_{5} \equiv h_{5}h_{4}h_{3}$.
\end{remark}

\section{Preliminary Results}\label{Preliminary Results}

The next two lemmas follow some ideas of Silva, see (\cite{Silva}, Lemma 4.5 and Lemma 4.6). The first proposition is a straightforward consequence of Proposition \ref{monomio}.

\begin{lemma}\label{lematecnico2}[Lemma 4.5, \cite{Silva}]
Let $\overline{m}(x_1,\dots,x_q)$ and $\overline{n}(x_1,\dots,x_q)$ be two monomials that start with the same variable. Let $m(x_1,\dots, x_q)$ and $n(x_1,\dots, x_q)$ be the monomials obtained from $\overline{m}$ and $\overline{n}$ respectively by deleting the first variable. If there exist matrices $A_{1,\alpha(x_{1})},\dots, A_{q,1,\alpha(x_{q})}$ such that
\begin{center}
$\overline{m}(A_{1,\alpha(x_{1})},\dots, A_{q,1,\alpha(x_{q})})$ and $\overline{n}(A_{1,\alpha(x_{1})},\dots, A_{q,1,\alpha(x_{q})})$
\end{center}
have in the same position the same non-zero entry, then
\begin{center}
$m(A_{1,\alpha(x_{1})},\dots, A_{q,1,\alpha(x_{q})})$ and $n(A_{1,\alpha(x_{1})},\dots, A_{q,1,\alpha(x_{q})})$
\end{center}
also have in the same position the same non-zero entry.
\end{lemma}

\begin{definition}
Let $m = x_{i_{1}}\ldots x_{i_{q}}$ be a monomial and let $\sigma \in S_{q}$ be a permutation. Let us define $m_{\sigma} = x_{i_{\sigma(1)}}\ldots x_{i_{\sigma(q)}}$. For any two integers $1 \leq k \leq l \leq q$, denote $m_{\sigma}^{[k,l]}$ the subword obtained from $m$ by discarding the first $k-1$ and the last $q - l$ factors:
\begin{center}
$m_{\sigma}^{[k,l]} = x_{i_{\sigma(k)}}.\cdots.x_{i_{\sigma(l)}}$.
\end{center}
\end{definition}

\begin{lemma}\label{essencial}
Let $m(x_1,\dots,x_q) = x_{i_{1}}\ldots x_{i_{l}}$ and $n(x_1,\dots,x_q) = x_{j_{1}}\ldots x_{j_{l}}$ be two monomials such that the matrices
\begin{center}
$n(A_{1,\alpha(x_{1})},\dots,A_{q,\alpha(x_{q})})$ and $m(A_{1,\alpha(x_{1})},\dots,A_{q,\alpha(x_{q})})$
\end{center}
have in the same position the same non-zero entry, then \[m(x_1,x_2\dots,x_q)\equiv n(x_1,x_2\dots,x_q) \mbox{ modulo }J.\]
\end{lemma}
\begin{proof}
According to the hypothesis,
\begin{center}
$n(A_{1,\alpha(x_{1})},\dots,A_{q,\alpha(x_{q})})$ and $m(A_{1,\alpha(x_{1})},\dots,A_{q,\alpha(x_{q})})$
\end{center}
have in the same position the same non-zero entry. Assume that this entry is $(i,j)$.

If $m$ and $n$ have only one variable, the proof is obvious. If $m$ and $n$ start with the same variable, the result follows from Lemma \ref{lematecnico2} and induction.

From now on, we will assume that $m$ and $n$ do not start with the same variable. Furthermore, we admit that they have more than one variable.

According to Proposition \ref{proposicaochave}, there exist matrix units

\begin{center}
$e_{a_{1}b_{1}} \in (M_{n}(K))_{\alpha(x_{i_{1}})},\ldots,e_{a_{l}b_{l}} \in (M_{n}(K))_{\alpha(x_{i_{l}})},$
\end{center}
\begin{center}
$e_{c_{1}d_{1}} \in (M_{n}(K))_{\alpha(x_{j_{1}})},\ldots,e_{c_{l}d_{l}} \in (M_{n}(K))_{\alpha(x_{j_{l}})}$
\end{center}
so that
\begin{center}
$e_{ij} = e_{a_{1}b_{1}}.\cdots.e_{a_{l}b_{l}} = e_{c_{1}d_{1}}.\cdots.e_{c_{l}d_{l}}$.
\end{center}
Moreover, there exists a permutation $\sigma \in S_{l}$ so that $e_{a_{\sigma(h)}b_{\sigma(h)}} = e_{c_{h}d_{h}}$ for all $h \in \widehat{l}$.

It is clear that $i = a_{1} = c_{1}, j = b_{l} = d_{l}$. In addition $b_{i} = a_{i+1}$, and $d_{i} = c_{i+1}$ for all $i \in \widehat{l-1}$.

According to Remark \ref{comentario}, we have $\sigma(1) > 1$. Let $t$ be the least integer such that $\sigma^{-1}(t+1) < \sigma^{-1}(1)$. Setting $p = \sigma^{-1}(t+1), r = \sigma^{-1}(1)$, and $s = \sigma^{-1}(t)$, we have $p < r \leq s$. Let us consider $n = n^{[1,p-1]}n^{[p,r-1]}n^{[r,s]}n^{[s+1,l]}$.

We can check that

\begin{enumerate}
\item $\alpha(n^{[1,p-1]}) = g_{a_{\sigma(1)}}^{-1}g_{a_{\sigma(p)}} = g_{c_{1}}^{-1}g_{a_{t+1}} = g_{a_{1}}^{-1}g_{a_{t+1}}$.
\item $\alpha(n^{[p,r-1]}) = g_{a_{\sigma(p)}}^{-1}g_{a_{\sigma(r)}} = g_{a_{t+1}}^{-1}g_{a_{1}}$.
\item $\alpha(n^{[r,s]}) = g_{a_{\sigma(r)}}^{-1}g_{b_{\sigma(s)}} = g_{a_{1}}^{-1}g_{b_{t}} = g_{a_{1}}^{-1}g_{a_{t + 1}}$.
\end{enumerate}

Let us define:

$\widetilde{m_{1}} = n^{[r,s]}n^{[1,p-1]}n^{[p,r-1]}n^{[s+1,l]}$ and $\widetilde{m_{2}} = n^{[r,s]}n^{[p,r-1]}n^{[1,p-1]}n^{[s+1,l]}$. The commuting variable of $\Omega$ associated to $e_{a_{1}b_{1}}$ is equal to the commuting variable associated to $e_{c_{r}d_{r}}$.
The monomials $m, \widetilde{m_{1}}, \widetilde{m_{2}}$ start with the same variable. Remind the Remark \ref{atencao}.

If $\alpha(n^{[1,p-1]}) =  \alpha(n^{[p,r-1]}) = \alpha(n^{[r,s]})$, then $m \equiv \widetilde{m}_{1} \ \ mod \ \ J$.
If $\alpha(n^{[1,p-1]}) =  (\alpha(n^{[p,r-1]}))^{-1} = \alpha(n^{[r,s]})$, then $m \equiv \widetilde{m}_{2} \ \ mod \ \ J$.

The result follows from Lemma \ref{lematecnico2} and induction.

\end{proof}

\section{An important result}

We now describe the polynomials that form a basis for the graded identities of the pair $(M_{n}(K), gl_{n}(K))$. 

\begin{theorem}
The sets $T_{G}(M_{n}(K),gl_{n}(K))$ and $J$ are equal.
\end{theorem}
\begin{proof}

According to Proposition \ref{identities} and Definition \ref{identidadesessenciais}, we have
\begin{center}
$J \subset T_{G}(M_{n}(K),gl_{n}(K))$.
\end{center}

We are going to prove that $T_{G}(M_{n}(K),gl_{n}(K)) \subset J $. According to Proposition \ref{matrizesgenericas}, the pairs $(M_{n}(K),gl_{n}(K))$ and
$(Gen,Gen^{(-)})$ have the same graded identities.

Assume, for the sake of contradiction that $T_{G}(M_{n}(K),gl_{n}(K)) \subsetneq J $. Thus, there exists

\begin{center}
$f(x_{1},\ldots,x_{l}) = \sum_{j=1}^{t} \lambda_{j} m_{j}(x_{1},\ldots,x_{l}) \in T_{G}(M_{n}(K),gl_{n}(K)) - J$,
\end{center}
where $\lambda_{j} \in K - \{0\}$ for all $j \in \widehat{t}$, and each $m_{j}$ is a monomial. By Corollary \ref{identidadesmonomiais}, we have $m_{j} \notin T_{G}(M_{n}(K),gl_{n}(K))$. By Proposition \ref{multi}, we can assume without loss of generality that $f$ is multi-homogeneous.


For convenience reasons, we assume that $t$ is
\begin{center}
$min\{k \in \mathbb{N}| \sum_{j=1}^{k} \eta_{j} n_{j}(x_{1},\ldots,x_{l}) \in T_{G}(M_{n}(K),gl_{n}(K)) - J\}$,
\end{center}
where $\eta_{j} \in K - \{0\}$ and $n_{j}$ is a monomial. Considering that $f \in T_{G}(M_{n}(K),gl_{n}(K))$, we have

\begin{center}
$f(A_{1,\alpha(x_{1})},\ldots,A_{l,\alpha(x_{l})}) = 0$.
\end{center}

Thus, there exists $i \in \{t\} - \{1\}$ such that
\begin{center}
$m_{1}(A_{1,\alpha(x_{1})},\ldots,A_{l,\alpha(x_{l})})$ and $m_{i}(A_{1,\alpha(x_{1})},\ldots,A_{l,\alpha(x_{l})})$
\end{center}
have in the same position the same non-zero entry. Hence, by Lemma \ref{essencial}, we have $m_{1} \equiv m_{i} \ \ mod \ \ J $.

Notice that $g = f - \lambda_{i}(m_{1} + m_{i})$ is equal to
\begin{center}
$(\lambda_{1} - \lambda_{i})m_{1} + \sum_{j=2}^{k}(\lambda_{j}m_{j}) - \lambda_{i}m_{i}  \in T_{G}(M_{n}(K),gl_{n}(K)) - J$.
\end{center}
This fact contradicts the minimality of $t$.

\end{proof}

\section{The $\mathbb{Z}_{3}$-graded identities of the pair $(M_{3}(K), gl_{3}(K))$}\label{Z3}

We now discuss about the $\mathbb{Z}_{3}$-graded case. From now on, $G$ denotes the abelian group $\mathbb{Z}_{3}$. Consider the endomorphisms of $K\langle X \rangle$ defined by the following rules.
In the first two endomorphisms, we consider $\alpha(x_{r}) = \overline{0}$. In the third, we consider $\alpha(x_{r+3}) = \overline{0}$.

\begin{enumerate}
\item For each integer $r \geq 2$, let $\mu_{r-1}(x_{i}) = x_{i}$, if $i \neq r - 1$, and $\mu_{r-1}(x_{r-1}) = [x_{r-1},x_{r}]$.
\item For each integer $r \geq 2$, let $\psi_{r+3}(x_{i}) = x_{i}$, if $i \neq r + 3$, and $\psi_{r+3}(x_{r+3}) = [x_{r},x_{r+1}]$.
\item For each integer $r \geq 4$, let $\rho_{r+1}(x_{i}) = x_{i}$, if $i \neq r + 1$, and $\psi_{r+1}(x_{r+1}) = [x_{2},x_{3}]$.
\end{enumerate}

The next Lemma is the immediate consequence of the well known identities $[ab,c] = a[b,c] + [a,c]b$ and $[a,bc] = [a,b]c + b[a,c]$, where $a,b,c \in K\langle X \rangle$.

\begin{lemma}
The following identity holds
\begin{center}
$[h_{1}h_{2},h_{3}h_{4}] = h_{1}h_{3}[h_{2},h_{4}] + h_{1}[h_{2},h_{3}]h_{4} + h_{3}[h_{1},h_{4}]h_{2} + [h_{1},h_{3}]h_{4}h_{2}$, where $h_{1},h_{2},h_{3},h_{4} \in K\langle X \rangle$.
\end{center}
\end{lemma}

\begin{corollary}\label{resultado1}
Let $h_{1},h_{2},h_{3}$ be monomials. Suppose that $\alpha(h_{1}) = \alpha(h_{2}) = \alpha(h_{3}) = \overline{0}$. The polynomial identity $[h_{1}h_{2},h_{3}]$ is consequence of polynomial identities
\begin{center}
$[h_{1},h_{3}]$, and $[h_{2},h_{3}]$.
\end{center}
\end{corollary}

\begin{lemma}\label{resultado4}
Let $h_{1},h_{2},h_{3},h_{4}$ be monomials such that $h_{1}h_{2}h_{3}h_{4}$ is multi-linear.
Suppose that $\alpha(h_{1}) = \alpha(h_{4}) = (\alpha(h_{2}))^{-1}$ and $\alpha(h_{3}) = \overline{0}$. Then $ h_{3}h_{4}h_{2}h_{1} - h_{1}h_{2}h_{3}h_{4}$ is consequence of the identity $h_{4}h_{2}h_{1} - h_{1}h_{2}h_{4}$ and identities of $J_{1}$.
\end{lemma}
\begin{proof}
Notice that $\alpha(h_{1}h_{2}) = \overline{0}$. Modulo $J_{1}$, we have $[h_{1}h_{2},h_{3}] \equiv 0$. Thus,
\begin{center}
$h_{3}h_{4}h_{2}h_{1} - h_{1}h_{2}h_{3}h_{4} \equiv h_{3}h_{4}h_{2}h_{1} - h_{3}h_{1}h_{2}h_{4} = h_{3}(h_{4}h_{2}h_{1} - h_{1}h_{2}h_{4})$.
\end{center}
\end{proof}

Following the ideas of the Lemma \ref{resultado4}, we can prove the next lemma.

\begin{lemma}\label{resultado4-parte2}
Let $h_{1},h_{2},h_{3},h_{4}$ be monomials such that $h_{1}h_{2}h_{3}h_{4}$ is multi-linear.
Suppose that $\alpha(h_{1}) = \alpha(h_{4}) = (\alpha(h_{2}))^{-1}$ and $\alpha(h_{3}) = \overline{0}$. Then $h_{1}h_{2}h_{3}h_{4} - h_{4}h_{2}h_{3}h_{1}$ is consequence of the identity $h_{1}h_{2}h_{4} - h_{4}h_{2}h_{1}$ and identities of $J_{1}$.
\end{lemma}

From now on, the hat in the above a polynomial described by the capital letters $Y,V,W$ indicates that the variable $x_{r}$ will be omitted
The hat over $x_{r}$ indicates that this variable is omitted too.


\begin{definition}
Let $h_{1} = \widehat{x_{r}}\cdots x_{r+s},h_{2} = x_{r+s+1}.\cdots.x_{r+s+t}$, where $r,t \geq 1$, $s \geq 0$ and $\alpha(x_{r}) = \overline{0}$. Let us define the following polynomials:
\begin{center}
$Y_{r,h_{1},h_{2}}  = [(x_{1}.\cdots.x_{r})h_{1},h_{2}],$
\end{center}
For $r \geq 2$:
\begin{center} $Y_{1,(x_{1}.\cdots.x_{r-1})h_{1},h_{2}} = [(x_{r})(x_{1}.\cdots.x_{r-1})h_{1},h_{2}].$
\end{center}
For $r \geq 3$:
\begin{center}
 $Y_{r-q,(x_{r-q}.\cdots.x_{r-1})h_{1},h_{2}} = [(x_{1}.\cdots.x_{r-q-1}x_{r})(x_{r - q}.\cdots.x_{r-1})h_{1},h_{2}], \newline 1 \leq q \leq r - 2$,
\end{center}

\end{definition}

\begin{lemma}\label{pioneiro}
For all $r \geq 2$, $Y_{r,h_{1},h_{2}} = Y_{r-1,x_{r-1}h_{1},h_{2}} + \mu_{r-1}(\widehat{Y_{r,h_{1},h_{2}}})$.
\end{lemma}
\begin{proof}
If $r \geq 3$, notice that
\begin{center}
$[x_{1}.\cdots.x_{r}.h_{1},h_{2}] = [x_{1}.\cdots.[x_{r-1},x_{r}]h_{1},h_{2}] +
[x_{1}.\cdots.x_{r}x_{r-1}h_{1},h_{2}]$.
\end{center}
We have already $r = 2$
\begin{center}
$[x_{1}.x_{2}.h_{1},h_{2}] = [[x_{1},x_{2}]h_{1},h_{2}] +
[x_{2}x_{1}h_{1},h_{2}] $.
\end{center}

In the two situations, the first summand on the right is $\mu_{r-1}(\widehat{Y_{r,h_{1},h_{2}}})$. The second is $Y_{r-1,x_{r-1}h_{1},h_{2}}$.
\end{proof}

\begin{corollary}\label{resultado2}
For all $r \geq 2$, $Y_{r,h_{1},h_{2}}$ is equal to
\begin{center}
$\mu_{r-1}(\widehat{Y_{r,h_{1},h_{2}}}) + \ldots + \mu_{1}(\widehat{Y_{2,x_{2}.\cdots.x_{r-1}h_{1},h_{2}}}) + Y_{1,x_{1}.\cdots.x_{r-1}h_{1},h_{2}}$.
\end{center}
\end{corollary}

\begin{definition}

Let $h_{1} = \widehat{x_{r}}.\cdots.x_{r+s}, h_{2} = x_{r+s+1}.\cdots.x_{r+s+t},\newline h_{3} = (x_{r+s+t+1}.\cdots.x_{r+s+t+u}),$ where $r,t,u \geq 1, s \geq 0$, and $\alpha(x_{r}) = \overline{0}$. Let us define
\begin{center}
$V_{r,h_{1},h_{2},h_{3}} = (x_{1}.\cdots.x_{r}.h_{1})h_{2}h_{3} - h_{3}h_{2}(x_{1}.\cdots.x_{r}.h_{1})$.
\end{center}

For $r \geq 2$:
\begin{center}
$V_{1,(x_{1}.\cdots.x_{r-1}h_{1},h_{2},h_{3})} = (x_{r})(x_{1}.\cdots.x_{r-1}h_{1})h_{2}h_{3} - h_{3}h_{2}(x_{r})(x_{1}.\cdots.x_{r-1}h_{1})$
\end{center}

For $r \geq 3$:
\begin{center}
$V_{r-q,(x_{r-q}.\cdots.x_{r-1})h_{1},h_{2},h_{3}} = (x_{1}.\cdots.x_{r-q-1}.x_{r})(x_{r-q}.\cdots.x_{r-1}h_{1})h_{2}h_{3} - \newline h_{3}h_{2}(x_{1}.\cdots.x_{r-q-1}.x_{r})(x_{r-q}.\cdots.x_{r-1}h_{1}), 1 \leq q \leq r - 2$.
\end{center}

\end{definition}

\begin{lemma}
For all $r \geq 2$, $V_{r,h_{1},h_{2},h_{3}} = V_{r-1,x_{r-1}h_{1},h_{2},h_{3}} + \mu_{r-1}(\widehat{V_{r,h_{1},h_{2},h_{3}}})$.
\end{lemma}
\begin{proof}
The proof is analogous to the proof of Lemma \ref{pioneiro}.



\end{proof}

\begin{corollary}\label{resultado2.1}
For all $r \geq 2$, $V_{r,h_{1},h_{2},h_{3}}$ is equal to
\begin{center}
$\mu_{r-1}(\widehat{V_{r,h_{1},h_{2},h_{3}}}) + \ldots + \mu_{1}(\widehat{V_{2,x_{2}.\cdots.x_{r-1}h_{1},h_{2},h_{3}}}) + V_{1,x_{1}.\cdots.x_{r-1}h_{1},h_{2},h_{3}}$.
\end{center}
\end{corollary}

\begin{definition}

Let $h_{1} = x_{r+s}.\cdots.\widehat{x_{r}}, h_{2} = x_{r+s+1}.\cdots.x_{r+s+t},\newline h_{3} = (x_{r+s+t+1}.\cdots.x_{r+s+t+u}),$ where $r,t,u \geq 1, s \geq 0$, and $\alpha(x_{r}) = \overline{0}$. Let us define
\begin{center}
$W_{r,h_{1},h_{2},h_{3}} = h_{2}(h_{1}.x_{r}.\cdots.x_{1})h_{3} -  h_{3}(h_{1}.x_{r}.\cdots.x_{1})h_{2}$.
\end{center}

For $r \geq 2$:
\begin{center}
$W_{1,h_{1}x_{r-1}.\cdots.x_{1},h_{2},h_{3}} =  h_{2}(h_{1}.x_{r-1}.\cdots.x_{1})(x_{r})h_{3} - h_{3}(h_{1}.x_{r-1}.\cdots.x_{1})(x_{r})h_{2}$
\end{center}

For $r \geq 3$:
\begin{center}
$W_{r-q,h_{1}x_{r-1}.\cdots.x_{r - q - 1},h_{2},h_{3}} = h_{2}(h_{1}.x_{r-1}.\cdots.x_{r - q})(x_{r}.x_{r-q-1}.\cdots.x_{1})h_{3} - \newline h_{3}(h_{1}.x_{r-1}.\cdots.x_{r - q})(x_{r}.x_{r-q-1}.\cdots.x_{1})h_{2}, 1 \leq q \leq r - 2$.
\end{center}
\end{definition}

\begin{lemma}
For all $r \geq 2$, $W_{r,h_{1},h_{2},h_{3}} = W_{r-1,h_{1}x_{r-1},h_{2},h_{3}} - \mu_{r-1}(\widehat{W_{r,h_{1},h_{2},h_{3}}})$.

\end{lemma}
\begin{proof}
The proof is analogous to the proof of Lemma \ref{pioneiro}.



\end{proof}

\begin{corollary}\label{resultado2.2}
For all $r \geq 2$, $W_{r,h_{1},h_{2},h_{3}}$ is equal to
\begin{center}
$- \mu_{r-1}(\widehat{W_{r,h_{1},h_{2},h_{3}}}) - \ldots - \mu_{1}(\widehat{W_{2,h_{1}x_{r-1}.\cdots.x_{2},h_{2},h_{3}}}) + W_{1,h_{1}x_{r-1}.\cdots.x_{2}x_{1},h_{2},h_{3}}$.
\end{center}
\end{corollary}


The next lemma can be easily verified.

\begin{lemma}\label{resultado3-auxiliar}
Let $(a_{1},a_{2},a_{3}) \in (\mathbb{Z}_{3}^{*})^{3}$. If $a_{1}+ a_{2}, a_{1}+a_{2}+a_{3} \neq \overline{0}$, then $a_{1} + a_{3} = a_{2} + a_{3} = \overline{0}$.
Analogously, if $a_{3} + a_{2}, a_{3} + a_{2} + a_{1} \neq \overline{0}$, then $a_{1} + a_{3} = a_{1} + a_{2} = \overline{0}$.
\end{lemma}

\begin{lemma}\label{resultado3}
Let $r \geq 2$ be an integer. Let $h_{1} = x_{1}.\cdots.x_{r}x_{r+1}x_{r+2}$ be a monomial such that

\begin{center}
  $\alpha(x_{r}),\alpha(x_{r+1}),\alpha(x_{r+2}), \alpha(x_{r+1}x_{r+2}), \alpha(x_{r}x_{r+1}x_{r+2}) \neq \overline{0}$.
\end{center}

Let
\begin{center}
  $h_{2} = x_{1}.\cdots.x_{r-1}x_{r+3}x_{r+2}$, where $\alpha(x_{r+3}) = \overline{0}$,
  $h_{3} = x_{1}.\cdots.x_{r+1}x_{r}x_{r+2}$.
\end{center}

So $\alpha(x_{r}) + \alpha(x_{r+1}) = \alpha(x_{r}) + \alpha(x_{r+2}) = \overline{0}$. Furthermore, $h_{1} = \psi_{r+3}(h_{2}) + h_{3}$.
\end{lemma}
\begin{proof}
According to the Lemma \ref{resultado3-auxiliar}, we have $\alpha(x_{r}) + \alpha(x_{r+2}) = \alpha(x_{r}) + \alpha(x_{r+1}) = \overline{0}$.

Notice that
\begin{center}
$h_{1} =  x_{1}.\cdots.[x_{r},x_{r+1}]x_{r+2} + h_{3}$.
\end{center}

The first summand above is $\psi_{r+3}(h_{2})$.

\end{proof}

Following the proof of Lemma \ref{resultado3}, we can prove the following.

\begin{lemma}\label{resultado5}
Let $r \geq 4$ be an integer. Let $h_{1} = x_{1}.x_{2}.x_{3}.\cdots.x_{r}$ be a monomial such that

\begin{center}
  $\alpha(x_{1}),\alpha(x_{2}),\alpha(x_{3}) , \alpha(x_{1}x_{2}), \alpha(x_{1}x_{2}x_{3}) \neq \overline{0}$.
\end{center}

Let
\begin{center}
  $h_{2} = x_{1}.x_{r+1}.x_{4}.\cdots.x_{r}$, where $\alpha(x_{r+1}) = \overline{0}$,
  $h_{3} = x_{1}.x_{3}.x_{2}.\cdots.x_{r}$.
\end{center}

So $\alpha(x_{2}) + \alpha(x_{3}) = \alpha(x_{1}) + \alpha(x_{3}) = \overline{0}$. Furthermore, $h_{1} = \rho_{r+1}(h_{2}) + h_{3}$.
\end{lemma}



\begin{definition}\label{base1}
Let $H_{1} = [h_{1},h_{2}]$ be an identity of type 1. We say that $H_{1}$ is reduced when $h_{1}, \mbox{and} h_{2}$ are monomials up to degree $3$.
\end{definition}

\begin{definition}\label{base2}
Let $H_{2} = h_{1}h_{2}h_{3} - h_{3}h_{2}h_{1}$ be an identity of type 2. We say that $H_{2}$ is reduced when $h_{1},h_{2}, \mbox{and} h_{3}$ are monomials up to degree $3$.
\end{definition}

\begin{definition}
Let $m_{1} = [h_{1},h_{2}]$ be an identity of type 1, and let $m_{2} = h_{1}h_{2}h_{3} - h_{3}h_{2}h_{1}$ be an identity of type 2. Let us define:
\begin{enumerate}
\item $degft_{1}(m_{1}) = deg h_{1},degft_{2}(m_{1}) = deg h_{2}$,
\item $degst_{1}(m_{2}) = deg h_{1}, degst_{2}(m_{2}) = deg h_{2},$ and $degst_{3}(m_{2}) = deg h_{3}$.
\end{enumerate}
\end{definition}

We know that $J_{1}$ is generated by the identities of type 1. In the next lemma, we prove that the weak $T_{G}$-ideal generated by the identities of
type 1 coincides with the weak $T_{G}$-ideal generated by the reduced identities of type 1.

\begin{theorem}
The weak $T_{G}$-ideal $J_{1}$ coincides with the weak $T_{G}$-ideal generated by the reduced polynomial identities of type 1.
\end{theorem}
\begin{proof}

Let $H_{1} = [h_{1},h_{2}]$ be an identity of type 1. We can prove it by double induction on $(degft_{1}(H_{1}),degft_{2}(H_{1}))$. The result is clear if $(degft_{1}(H_{1}),degft_{2}(H_{1})) = (3,3)$. Suppose that the result holds if $(degft_{1}(H_{1}),degft_{2}(H_{1})) = (3,3),\ldots,(m,3)$.

To analyze the case $degft_{1}(H_{1}) = m+1$, let us first consider that $h_{1} = h_{3}h_{4}$, where $h_{3}$ and $h_{4}$ are monomials such that $\alpha(h_{3}) = \alpha(h_{4}) = \overline{0}$. So $[h_{3}h_{4},h_{2}] = h_{3}[h_{4},h_{2}] + [h_{3},h_{2}]h_{4}$. This part of the result follows from Corollary \ref{resultado1}.


If $h_{1}$ cannot be rewritten in this way, there are two cases to analyze.
\begin{enumerate}
\item Assume that $h_{1}$ has a variable of $G$-degree $\overline{0}$. According to Corollary \ref{resultado2}, we can rewrite $[h_{1},h_{2}]$ as $[h_{3},h_{2}] + p$. Here $p$ is a sum of consequences of identities of type 1. These identities are of type $[\widetilde{h_{i}},h_{2}]$, with $deg (\widetilde{h_{i}}) = m - 1$. In $[h_{2},h_{3}]$, we have $deg(h_{1}) = deg(h_{3})$, $h_{3} = xh_{4}$, where $xh_{4}h_{2}$ is a multi-linear monomial and $\alpha(x) = \alpha(h_{4}) = \overline{0}$. So $[xh_{4},h_{2}] = x[h_{4},h_{2}] + [x,h_{2}]h_{4}$.

\item Suppose that $h_{1}$ does not have a variable of $G$-degree $\overline{0}$. By Lemma \ref{resultado5}, there exist multi-linear monomials $h_{3}$ and $h_{4}$, where $deg(h_{3}) = m$ and $deg(h_{4}) = m-1$. The monomial $h_{1}$ is a sum of $h_{3}$ and a consequence of $h_{4}$. Denote this consequence of $h_{4}$ by $h_{5}$. The monomial $h_{2}h_{3}$ is multi-linear. Without loss of generality, we can assume that $h_{2}h_{4}$ is multi-linear too. Note that $\alpha(h_{3}) = \alpha(h_{4}) = \overline{0}$. By Lemma \ref{resultado5} again, there exist two multi-linear monomials $h_{6}$ and $h_{7}$ such that $h_{3} = h_{6}h_{7}$, where $deg(h_{6}) = 2$ and $\alpha(h_{6})=\alpha(h_{7}) = \overline{0}$.
Thus $[h_{1},h_{2}] = [h_{6}h_{7} + h_{5},h_{2}] = h_{6}[h_{7},h_{2}] + [h_{6},h_{2}]h_{7} + [h_{5},h_{2}]$. The identity $[h_{5},h_{2}]$ is consequence of $[h_{4},h_{2}]$.
\end{enumerate}
Assume that the result is true for $(deg(h_{1}),deg(h_{2})) = (m,3),\ldots,(m,r)$. By following the last arguments word by word, we can prove the result holds for
\begin{center}
$(deg(h_{1}),deg(h_{2})) = (m,r + 1)$.
\end{center}
The proof is complete.
\end{proof}

\begin{theorem}
The weak $T_{G}$-ideal $J_{2}$ coincides with the weak $T_{G}$-ideal generated by the reduced polynomial identities of type 2.

\end{theorem}

\begin{proof}

Let $H_{2} = h_{1}h_{2}h_{3} - h_{3}h_{2}h_{1}$ be an identity of type 2.

We prove by triple induction on $(degst_{1}(H_{2}),degst_{2}(H_{2}) , degst_{3}(H_{2}))$. The result is immediate if $(degst_{1}(H_{2}),degst_{2}(H_{2}) , degst_{3}(H_{2})) = (3,3,3)$. Suppose that the result is true for $(degst_{1}(H_{2}),degst_{2}(H_{2}) , degst_{3}(H_{2})) = (3,3,3),\ldots,(m,3,3)$.

We assume first that there exist two monomials $h_{4}$ and $h_{5}$ such that $h_{1} = h_{4}h_{5}$, where $\alpha(h_{4}) = \overline{0}$. By Lemma \ref{resultado4}, $H_{2}$ follows from the identity of type 2 $(h_{5}h_{2}h_{3} - h_{3}h_{2}h_{5})$ and elements of $J_{1}$. Hereafter, we assume that $h_{1}$ cannot be decomposed in this way.

Now, suppose that $h_{1}$ has a variable of $G$-degree $\overline{0}$. Due to Corollary \ref{resultado2.1}, we have $h_{1}h_{2}h_{3} - h_{3}h_{2}h_{1}$ is $xh_{0}h_{2}h_{3} - h_{3}h_{2}xh_{0} + p$, where $x \in X$, $\alpha(x) = \overline{0}$. Furthermore $xh_{0}h_{2}h_{3}$ is a multi-linear monomial. The polynomial $p$ is a sum of consequences of identities of type 2 of type $\widetilde{h_{4}}h_{2}h_{3} - h_{3}h_{2}\widetilde{h_{4}}$. Here $\widetilde{h_{4}}$ is a monomial and $deg(\widetilde{h_{4}}) = m -1$. The identity $\widetilde{h_{4}}h_{2}h_{3} - h_{3}h_{2}\widetilde{h_{4}}$ is type 2. According to the Lemma \ref{resultado4}, $xh_{0}h_{2}h_{3} - h_{3}h_{2}xh_{0}$ is consequence of $h_{0}h_{2}h_{3} - h_{3}h_{2}h_{0}$ and elements of $J_{1}$. From now on, we assume that $h_{1}$ does not have a variable of $G$-degree $\overline{0}$.

By Lemma \ref{resultado5}, there exist multi-linear monomials $h_{6},h_{7}$, with $deg(h_{6}) = m$ and $deg(h_{7}) = m-1$, such that $h_{1}$ is a sum of $h_{6}$ and a consequence of $h_{7}$. Moreover, $h_{6} = h_{8}h_{9},$ where $h_{8}$ and $h_{9}$ are monomials and $\alpha({h_{8}}) = \overline{0}$. Notice that $\alpha(h_{1}) = \alpha(h_{6}) = \alpha(h_{9})$. The monomial $h_{8}h_{9}h_{2}h_{3}$ is multi-linear. Moreover, $h_{8}h_{9}h_{2}h_{3} - h_{2}h_{3}h_{8}h_{9}$ is consequence of identity of type 2, $h_{9}h_{2}h_{3} - h_{2}h_{3}h_{9}$, and elements of $J_{1}$. Let $h_{10}$ be the consequence of $h_{7}$ reported before. Without loss of generality, we can assume that $h_{7}h_{2}h_{3}$ is multi-linear.
The polynomial identity $h_{10}h_{2}h_{3} - h_{3}h_{2}h_{10}$ is a consequence of the polynomial identity of type 2 $h_{7}h_{2}h_{3} - h_{3}h_{2}h_{7}$. Considering $h_{1}h_{2}h_{3} - h_{3}h_{2}h_{1} = (h_{8}h_{9}h_{2}h_{3} - h_{2}h_{3}h_{8}h_{9}) + (h_{10}h_{2}h_{3} - h_{3}h_{2}h_{10})$, the result of this part of theorem follows.

 Assume that the result is valid for 
 
 \begin{center}
 $(degst_{1} (h_{1}h_{2}h_{3}), degst_{2} ( h_{1}h_{2}h_{3}), degst_{3} (h_{1}h_{2}h_{3})) = (m,3,3), \ldots, (m,k,3)$.
  \end{center}
 
 Suppose that there exists monomials $h_{4}$ and $h_{5}$ such that $h_{2} = h_{4}h_{5}$, where $\alpha(h_{5}) = \overline{0}$. By Lemma \ref{resultado4-parte2}, $H_{2}$
follows from $h_{1}h_{4}h_{3} - h_{3}h_{4}h_{1}$ and elements of $J_{1}$. Hereafter, we assume that $h_{2}$ cannot be decomposed in this way.

 Now, we assume that there exists a variable of $G$-degree $\overline{0}$ in $h_{2}$. According to Lemma \ref{resultado2.2}, $h_{1}h_{2}h_{3} - h_{3}h_{2}h_{1}$ is $h_{1}h_{4}xh_{3} - h_{3}h_{4}xh_{1} + p$, where $x$ is a variable with $\alpha(x) = \overline{0}$, $h_{3}h_{4}xh_{1}$ is a multi-linear monomial, and $\alpha(h_{4}) = \alpha(h_{2})$. The polynomial $p$ is a sum of consequences of identities of type 2. These identities are of the type $h_{1}\widetilde{h_{5}}h_{3} - h_{3}\widetilde{h_{5}}h_{1}$, where $h_{1}\widetilde{h_{5}}h_{3}$ is a multi-linear monomial with $deg (\widetilde{h_{5}}) = k - 1$ and $\alpha(h_{2}) = \alpha(\widetilde{h_{5}})$. At light of Lemma \ref{resultado4-parte2}, $h_{1}h_{4}xh_{3} - h_{3}h_{4}xh_{1}$ is a consequence of $(h_{1}h_{4}h_{3} - h_{3}h_{4}h_{1})$ and elements of $J_{1}$. Henceforth, we suppose that $h_{2}$ does not have a variable of $G$-degree $\overline{0}$.

 According to Lemma \ref{resultado3}, $h_{2}$ is equal to a sum of a monomial $h_{4}$ and a consequence of a monomial $h_{5}$. Let $h_{6}$ be this consequence of $h_{5}$. We have $deg(h_{5}) = k - 1, deg(h_{4}) = k$, $\alpha(h_{2}) = \alpha(h_{4}) = \alpha(h_{5})$. The monomial $h_{1}h_{4}h_{3}$ is multi-linear. Without loss of generality, we can assume that $h_{1}h_{5}h_{3}$ is multi-linear too.
By Lemma \ref{resultado3} again, we have $h_{4} = h_{7}h_{8}$, where $\alpha(h_{8}) = \overline{0}$. Thus $h_{1}h_{2}h_{3} - h_{3}h_{2}h_{1} = (h_{1}h_{7}h_{8}h_{3} - h_{3}h_{7}h_{8}h_{1}) + (h_{1}h_{6}h_{3} - h_{3}h_{6}h_{1})$.
According to the Lemma \ref{resultado4-parte2}, $(h_{1}h_{7}h_{8}h_{3} - h_{3}h_{7}h_{8}h_{1})$ is a consequence of $h_{1}h_{7}h_{3} - h_{3}h_{7}h_{1}$ and elements of $J_{1}$. The summand $(h_{1}h_{6}h_{3} - h_{3}h_{6}h_{1})$ is a consequence of $(h_{1}h_{5}h_{3} - h_{3}h_{5}h_{1})$.

Lastly, assume that the result holds for
\begin{center}
$(degst_{1} (H_{2}), degst_{2} (H_{2}), degst_{3} (H_{2}) = (m,k,3),\ldots, (m,k,r)$.
\end{center}
By mimicking of the arguments of the case $(degst_{1}(H_{2}), degst_{2} (H_{2}), degst_{3} (H_{2}) = (m+1,3,3)$ and Corollary \ref{resultado2.2}, we can prove the result is true for
\begin{center}
$(degst_{1}(H_{2}), degst_{2}(H_{2}), degst_{3}(H_{2})) = (m,k,r + 1)$.
\end{center}

The proof is complete. The conclusion follows by induction.
\end{proof}

\begin{corollary}\label{conclusao}
The weak $T_{G}$-ideal $J$ is generated by the polynomials reported in Definitions \ref{base1} and \ref{base2}.
\end{corollary}

\end{document}